\documentclass[a4paper, 11pt]{amsart}
 \usepackage{amsaddr}
 
\usepackage{fullpage}

\usepackage{amsmath}
\usepackage{amsthm}
\usepackage{amsfonts}
\usepackage{amssymb}
\usepackage{tikz-cd}
\usepackage[utf8x]{inputenc}
\usepackage{enumerate}
\usepackage[english]{babel}
\usepackage{esint}
\usepackage{hyperref}
\usepackage{xcolor}
\usepackage{bm}
\usepackage{bbm}
\usepackage[mathscr]{eucal}

\theoremstyle{plain}
\newtheorem{theorem}{Theorem}[section]
\newtheorem{lemma}[theorem]{Lemma}

\theoremstyle{definition}

\theoremstyle{remark}

\newtheorem{remark}[theorem]{Remark}

\numberwithin{equation}{section}

\newcommand{\vr}{\varrho}
\newcommand{\vu}{\textbf{\textup{u}}}
\newcommand{\vv}{\textbf{\textup{v}}}
\newcommand{\vt}{\vartheta}

\newcommand{\vq}{\textbf{q}}

\newcommand{\vw}{\textbf{w}}
\newcommand{\ve}{\varepsilon}
\newcommand{\vS}{\mathbb{S}}

\newcommand{\vf}{\textbf{f}}

\newcommand{\Div}{{\rm div}_x}
\newcommand{\Grad}{\nabla_x}
\newcommand{\DerTime}{\partial_t}
\newcommand{\dx}{\textup{d}x}

\newcommand{\dt}{\textup{d}t}

\definecolor{pinegreen}{rgb}{0.0, 0.47, 0.44}

\def\softd{{\leavevmode\setbox1=\hbox{d}%
		\hbox to 1.05\wd1{d\kern-0.4ex{\char039}\hss}}}

\allowdisplaybreaks[1]

\makeatletter
\renewcommand{\email}[2][]{%
	\ifx\emails\@empty\relax\else{\g@addto@macro\emails{,\space}}\fi%
	\@ifnotempty{#1}{\g@addto@macro\emails{\textrm{(#1)}\space}}%
	\g@addto@macro\emails{#2}%
}
\makeatother

	\title{On a blow-up criterion for the  Navier--Stokes--Fourier system under general equations of state }
	
\author{Anna Abbatiello}
\address{University of Campania ``L.~Vanvitelli", Department of Mathematics and Physics, \\Viale A.~Lincoln 5, 81100 Caserta, Italy.}
\email{anna.abbatiello@unicampania.it}

\author{Danica Basari\'{c}}
\address{Institute of Mathematics of the Academy of Sciences of the Czech Republic, \\ \v Zitn\' a 25, 115 67 Praha 1, Czech Republic.}
 \email{basaric@math.cas.cz}
 
\author{Nilasis Chaudhuri}
\address{University of Warsaw, Faculty of Mathematics, Informatics and Mechanics, \\
Stefana Banacha 2, 02-097 Warsaw, Poland.} 
\email{nchaudhuri@mimuw.edu.pl}

\date{}

\begin{document}
	
	\maketitle
	
	\begin{abstract}
		In this paper we prove a blow-up criterion for the compressible Navier-Stokes-Fourier system for general thermal and caloric equations of state with inhomogeneous boundary conditions for the velocity and the temperature.  Assuming only that Gibb's equation and the thermodynamic stability hold, we show that solutions in a certain regularity class remain regular under the condition that the density, the temperature and the modulus of the velocity are bounded.
	\end{abstract}

	\section{Introduction}
A conditional regularity criterion  for solutions to  a system of partial differential equations in fluid mechanics is a condition involving lower-order norms which, if satisfied, implies that the solutions remains regular; in particular,  it can be applied to show that a local strong solution can be extended  beyond its maximal time of existence. A direct consequence of the aforementioned result is a blow-up criterion meaning that if a blow-up of solutions occurs then some lower-order norms are not bounded. The goal of this paper is to provide a conditional regularity result for solutions to  the Navier-Stokes-Fourier system describing the motion of a compressible viscous and heat-conducting fluid,
\begin{subequations}\label{problem}
	\begin{align}
		\DerTime \vr + \Div (\vr \vu)&=0, \label{continuity equation}\\
		\DerTime (\vr \vu) + \Div (\vr \vu \otimes \vu) + \nabla_x p(\vr, \vt) &= \Div \vS (\mathbb{D}_x \vu) + \vr \vf, \label{balance of momentum}\\
		\DerTime \big( \vr e(\vr, \vt) \big) + \Div \big(\vr e(\vr, \vt) \vu \big) + \Div \vq(\Grad \vt) &= \vS (\mathbb{D}_x \vu) : \mathbb{D}_x \vu - p(\vr, \vt) \Div \vu. \label{balance internal energy}
	\end{align}
	Here, the unknown quantities are the density $\vr=\vr(t,x)$, the velocity $\vu=\vu(t,x)$ and the absolute temperature $\vt= \vt(t,x)$ of the fluid. 
	
	At any time $t \in (0,T)$, we suppose that the fluid is confined to a bounded domain $\Omega \subset \mathbb{R}^3$ with impermeable boundary, where the temperature and the (tangential) velocity are given on $\partial \Omega$,
	\begin{align}
		\vu|_{\partial \Omega} &= \vu_B, \ \vu_B=\vu_B(x), \quad \vu_B \cdot \textbf{n}=0, \label{boundary velocity}\\
		\vt|_{\partial \Omega} &= \vt_B, \ \vt_B=\vt_B(x), \quad \vt_B\geq \underline{\vt}_B >0. \label{boundary temperature}
	\end{align}
		
	The problem is closed prescribing the initial data 
	\begin{equation} \label{initial data}
		\vr(0,\cdot) =\vr_0, \quad \vu(0,\cdot)=\vu_0, \quad \vt(0,\cdot)=\vt_0.
	\end{equation}
\end{subequations}	
	For definiteness, we extend the boundary data $(\vu_B, \vt_B)$ as $(\widetilde{\vu}, \widetilde{\vt})$, where the latter are the unique solutions of the following problems,
	\begin{align}
		\Div \vS(\mathbb{D}_x \widetilde{\vu}) &=0 \ \mbox{in } \Omega, \quad \widetilde{\vu}|_{\partial \Omega}=\vu_B, \label{extension u_B}\\
		\Delta_x \widetilde{\vt}&=0 \ \mbox{in } \Omega, \quad \widetilde{\vt}|_{\partial \Omega}=\vt_B; \label{extension vt_B}
	\end{align} 
	to simplify notation, we will denote the extensions $(\widetilde{\vu}, \widetilde{\vt})$ as $(\vu_B, \vt_B)$.

	\subsection{Constitutive relations} \label{Constitutive relations}
	We will suppose that the fluid is Newtonian, meaning that the viscous stress tensor $\vS$ satisfies Newton's rheological law
	\begin{equation} \label{viscosity}
		\vS(\mathbb{D}_x \vu) = 2\mu \left[\mathbb{D}_x \vu -\frac{1}{3}\Div \vu \mathbb{I}\right] + \eta\, \Div \vu \,\mathbb{I}, \quad \mu>0, \ \eta > 0,
	\end{equation}
	while the heat flux obeys Fourier's law
	\begin{equation}
		\vq(\Grad \vt) = -\kappa \Grad \vt, \quad \kappa>0.
	\end{equation}
	The viscosity coefficients $\mu, \eta$ as well as the heat conductivity coefficient $\kappa$ are constants.

	We suppose that the pressure $p=p(\vr, \vt)$ and the internal energy $e=e(\vr, \vt)$ are related through  the following relation
	\begin{equation} \label{pressure and internal energy}
		\begin{array}{c}\displaystyle \vspace{6pt}
			\frac{\partial p}{\partial \vt }(\vr, \vt) \simeq \vr \frac{\partial e}{\partial \vt }(\vr, \vt), \\ \displaystyle \vspace{6pt}
			\mbox{ i.e. there are constants } c_1, c_2>0 \mbox{ such that } c_1 \vr \, \frac{\partial e}{\partial \vt } \leq \frac{\partial p}{\partial \vt } \leq c_2 \vr \, \frac{\partial e}{\partial \vt }.
		\end{array}	
	\end{equation}
	Moreover, thermodynamic stability holds for the latter, i.e.
	\begin{equation} \label{thermodynamic stability}
		\frac{\partial p}{\partial \vr}(\vr, \vt) >0, \quad c_{\nu}(\vr, \vt):= \frac{\partial e}{\partial \vt}(\vr, \vt)>0 \quad \mbox{for } \vr, \vt>0.
	\end{equation}
	Finally, we suppose that they are related to a third quantity, the entropy $s=s(\vr, \vt)$, through Gibb's relation
	\begin{equation} \label{Gibbs relation}
		\vt Ds = De + p D \left(\frac{1}{\vr}\right) \quad \mbox{for } \vr, \vt>0
	\end{equation}
	where the symbol $D$ stands for the differentiation with respect  to the density $\vr$ and the temperature $\vt$.
	In particular, from \eqref{Gibbs relation} it is straightforward to deduce Maxwell's relation
	\begin{equation} \label{Maxwell relation} 
		\vt \frac{\partial p}{\partial \vt} = - \vr^2 \frac{\partial e}{\partial \vr} + p = - \vt \vr^2 \frac{\partial s}{\partial \vr} \quad \mbox{for } \vr, \vt>0.
	\end{equation}

The Navier-Stokes-Fourier system  with inhomogeneous boundary conditions \eqref{problem} is well-posed in the class of strong solutions on a local time interval as established by Valli and Zajaczkowski \cite{ValZaj} (see also Valli \cite{Vall2}, \cite{Vall1}). However it is still an open question whether a global solution exists or not, except certain special cases (as under the assumption of small data, see Matsumura and Nishida \cite{MatNis}, Valli and Zajaczkowski \cite{ValZaj}, or 
relaxing the concepts of solutions e.g. weak solutions or variational solutions, see \cite{ChauFei}, Feireisl and Novotn\'{y} \cite{FeiNov}).
	The aim of this study is to provide a conditional regularity criterium for system \eqref{problem} with general
constitutive equations \eqref{pressure and internal energy} suitable for the description of real materials. Besides a general equation of state as \eqref{pressure and internal energy}  we suppose only that the thermodynamic stability and the Gibb's equation hold. \\
	Recently in \cite{BasFeiMiz} the authors proved a blow-up criterion for  system \eqref{problem}  in the case that the equation of state for the pressure $p$ and the internal energy $e$ is given by the standard Boyle-Mariotte law of a perfect gas
	\begin{equation} \label{Boyle Mariotte law}
		p(\vr,\vt) = \vr\vt,  \  e(\vr, \vt) = c_{\nu} \vt, \ c_{\nu}>0.
	\end{equation}
	The specific heat at constant volume $c_{\nu}$ is a positive constant. The authors in \cite{BasFeiMiz} show that if  the solution $(\vr, \vu, \vt)$, in the class of regularity given by  Valli \cite{Vall1},  is bounded then it remains regular. In other words if the maximal existence time $ T_{\rm max}$ is finite then the $L^\infty$-norm of 	$(\vr, \vu, \vt)$ blows up when approaching $ T_{\rm max}$. Interestingly in \cite{BasFeiMiz} they consider Dirichlet boundary conditions as the majority of results concerning conditional regularity for system \eqref{problem}  were obtained on a bounded fluid domain with conservative boundary conditions, i.e. 
	$$ \vu_{|\partial\Omega} = 0, \ \Grad\vt\cdot\textbf{n}_{|\partial\Omega} = 0;$$
cf. \cite{FaJiOu}, \cite{FeWeZh}, \cite{Huang}, \cite{SunWanZha1}, \cite{WenZhu1}, \cite{WenZhu2}. As motivated in \cite{BasFeiMiz} in view of possible applications, for instance
	numerical implementations of problems of the real world,	we focus on open systems driven by inhomogeneous boundary conditions in a general thermodynamic framework. The main novelty of our work is to investigate a regularity criterion that is valid for a large class of thermodynamical models provided they are stable. We point out that relation \eqref{pressure and internal energy} describes most of the physically relevant situations, including not only the Boyle-Mariotte law \eqref{Boyle Mariotte law} but also the general equation of state considered in \cite[Section 4.1.1]{FeiNov}.
	\\
	It is remarkable that, as pointed out in \cite{BasFeiMiz},  conditional regularity results imply \emph{stability} with respect to the data given in the regularity class needed for the local-in-time existence.  \\
	The paper is organized as follows. In Section \ref{loc} we introduce the class of regular solutions to system \eqref{problem} whose local-in-time existence is proven in \cite{ValZaj}. In Section \ref{main} we state our main result concerning conditional regularity: if the strong solution is bounded, then it exists globally in time. Then the rest of the paper is devoted to the proof of the main result. In particular in Section \ref{estim} we establish the estimates involving the material time derivative of the velocity and of the temperature. Next, in Section \ref{posit} we prove the minimum principle for both the density and the temperature. We conclude employing the $L^p-L^q$ regularity to the temperature equation.

	\section{Local existence of strong solutions}\label{loc}
	
	We  are going to state the existence result established by Valli and Zajaczkowski \cite{ValZaj}; in particular, we highlight the necessary hypotheses in the following.
	\begin{itemize}
		\item[(i)] \textit{Domain.} $\Omega \subset \mathbb{R}^3$ is a bounded domain with $\partial \Omega$ of class $C^3$.
		\item[(ii)] \textit{Prescribed data.} The initial and boundary data satisfy
		\begin{align} 
			\vr_0 &\in W^{2,2}(\Omega), \quad 0< \underline{\vr}_0 \leq \vr_0 \leq \overline{\vr}_0, \label{rc1}\\
			\vu_0- \vu_B &\in W_0^{3,2}(\Omega; \mathbb{R}^3),\\
			\vt_0-\vt_B &\in W_0^{3,2}(\Omega), \quad 0< \underline{\vt}_0 \leq \vt_0 \leq \overline{\vt}_0, \label{rc3}\\
			\vu_B &\ \in W^{3,2}(\Omega; \mathbb{R}^3), \quad \vu_B\cdot \textbf{n} = 0,\\
			\vt_B & \ \in W^{3,2}(\Omega), \quad 0< \underline{\vt}_B \leq \vt_B. \label{rc6}
		\end{align} 
		The driving force $\vf = \vf(x)$ is independent of time and satisfies 
		\begin{equation} \label{rc7}
			\vf \in W^{1,2}_0(\Omega; \mathbb{R}^3).
		\end{equation}
		\item[(iii)] \textit{Compatibility condition.} We introduce the following quantities
		\begin{align}
			\dot{\vu}_0&:= \frac{1}{\vr_0} \left[ - \Grad p(\vr_0, \vt_0) + \Div \vS (\mathbb{D}_x \vu_0)\right] - (\vu_0 \cdot \Grad ) \vu_0 + \vf, \\
			\dot{\vt}_0 &:= \frac{1}{\vr_0 \ c_{\nu}(\vr_0, \vt_0)} \left[ \kappa \Delta_x \vt_0 + \vS (\mathbb{D}_x \vu_0) : \mathbb{D}_x \vu_0 - \vt_0 \frac{\partial p}{\partial \vt }(\vr_0, \vt_0) \Div \vu_0\right] - \vu_0 \cdot \Grad \vt_0,  
		\end{align}
		and we assume that 
		\begin{equation} \label{compatibility condition}
			(\dot{\vu}_0, \dot{\vt}_0) \in W^{1,2}_0(\Omega; \mathbb{R}^4).
		\end{equation}
	\end{itemize}

We set 
	\begin{equation*}
			\mathfrak{D}_0:= \max \left\{ \mu, \eta, \kappa, \| \vr_0 \|_{W^{2,2}(\Omega)}, \| (\vu_0, \vt_0, \vu_B, \vt_B)\|_{W^{3,2}(\Omega; \mathbb{R}^8)}, \ \| \textbf{\textup{f}}\|_{W^{1,2}_0(\Omega; \mathbb{R}^3)}, \frac{1}{\underline{\vr}_0}, \frac{1}{\underline{\vt}_0}, \frac{1}{\underline{\vt}_B} \right\}.
		\end{equation*}

	We are now ready to state the local existence result which can be found in \cite[Theorem 2.5]{ValZaj}. 
	
	\begin{theorem} \label{Local existence}
		Let $\Omega \subset \mathbb{R}^3$ be a bounded domain with $\partial \Omega$ of class $C^3$. Let the initial data $(\vr_0, \vu_0, \vt_0)$ the boundary data  $(\vu_B, \vt_B)$ and the external force \textbf{\textup{f}} belong to the regularity class \eqref{rc1}--\eqref{rc7} and satisfy the compatibility condition \eqref{compatibility condition}. Moreover, let the pressure $p=p(\vr, \vt)$ and the internal energy $e=e(\vr, \vt)$ be $C^2$-functions of $(\vr, \vt)$.
		
		Then there exists a positive time $T_{\rm max}$ such that the Navier--Stokes--Fourier system \eqref{problem} with the constitutive relations \eqref{viscosity}--\eqref{Gibbs relation} admits a unique solution $(\vr, \vu, \vt)$ in $[0,T_{\rm max}) \times \overline{\Omega}$ such that for any $T< T_{\rm max}$ they belong to the regularity class
		\begin{align}
			\vr & \in C([0,T]; W^{2,2}(\Omega)) \cap C^1([0,T]; W^{1,2}(\Omega)), \label{rc4}\\
			(\vu, \vt) & \in L^2(0,T; W^{3,2}(\Omega; \mathbb{R}^4 )) \cap C([0,T]; W^{2,2}(\Omega; \mathbb{R}^4)). \label{rc5}
		\end{align} 
	\end{theorem}

	\section{Main result: blow-up criterion and conditional regularity}\label{main}
	
	\begin{theorem} [Blow-up criterion] \label{Blow-up Theorem}
		Let the hypotheses of Theorem \ref{Local existence} hold and let $(\vr, \vu, \vt)$ be the unique strong solution on the time-interval $[0,T_{\rm max})$, whose existence is stated in Theorem \ref{Local existence}. If $T_{\rm max}$ is  finite, then
		\begin{equation*}
			\lim_{\tau \to T_{\rm \max}^-}\| (\vr, \vu, \vt) (\tau, \cdot) \|_{L^{\infty}(\Omega; \mathbb{R}^5)} = \infty.
		\end{equation*}
	\end{theorem}

	To get Theorem \ref{Blow-up Theorem}, we prove its contrapositive.
	
	\begin{theorem} [Conditional regularity] \label{Conditional regularity}
		Under the hypotheses of Theorem \ref{Blow-up Theorem}, let $(\vr, \vu, \vt)$ be the unique strong solution on the time-interval $[0,T_{\rm max})$, whose existence is stated in Theorem \ref{Local existence}. Moreover, let us suppose that for any $T < T_{\rm max}$ 
		\begin{equation} \label{boundedness}
	\begin{split}		
			&\sup_{(\tau,x) \in [0,T] \times \overline{\Omega}}\vr(\tau,x) \leq \overline{\vr}, \ 	\sup_{(\tau,x) \in [0,T] \times \overline{\Omega}} |\vu(\tau,x)| \leq \overline{u}, \ \sup_{(\tau,x) \in [0,T] \times \overline{\Omega}}\vt(\tau,x) \leq   \overline{\vt}.
	\end{split}	\end{equation}
		Then there exist positive constants 
		\begin{equation*}
			(\underline{\vr}, \underline{\vt}, C_0) = (\underline{\vr}, \underline{\vt}, C_0)(T_{\rm \max};  \overline{\vr}, \overline{u}, \overline{\vt}; \mathfrak{D}_0)>0, 
		\end{equation*}
		such that for any $T< T_{\rm max}$
		\begin{equation} \label{bound from below density and temperature}
			\inf_{(\tau,x) \in [0,T]\times \overline{\Omega}}  \vr(\tau, x) \geq \underline{\vr}, \ \inf_{(\tau,x) \in [0,T]\times  \overline{\Omega}}  \vt(\tau, x)  \geq \underline{\vt},
		\end{equation}
		\begin{equation} \label{final estimate}
				\sup_{\tau\in [0,T]} \left\{ \| \vr(\tau, \cdot) \|_{W^{2,2}(\Omega)}, \| (\vu-\vu_B, \vt-\vt_B)(\tau, \cdot)\|_{W^{3,2}_0(\Omega; \mathbb{R}^4)},  \| (\DerTime \vu, \DerTime \vt)(\tau, \cdot)\|_{W^{1,2}_0(\Omega; \mathbb{R}^4)} \right\}\leq C_0.
		\end{equation}
	\end{theorem}
	
	\begin{remark}
		Proceeding as in \cite[Theorem 1.2]{FaZiZh}, if $T_{\rm max}$ is finite, estimates \eqref{bound from below density and temperature} and \eqref{final estimate}  in particular imply that 
		\begin{equation*}
			(\vr, \vu, \vt) (T_{\rm max}, \cdot) = \lim_{\tau \to T_{\rm \max}^-} 	(\vr, \vu, \vt) (\tau, \cdot) 
		\end{equation*}
		can be taken as the new initial data to problem \eqref{continuity equation}--\eqref{balance internal energy} and therefore, that the solution $(\vr, \vu, \vt)$ can be extended.
	\end{remark}

	\section{Estimates}\label{estim}
	
	Our goal in this section is to derive \textit{a priori} bounds in higher order norms depending solely on $T_{\rm max}$, the constants $(\overline{\vr}, \overline{u}, \overline{\vt})$ and the prescribed data but independent of the choice $T< T_{\rm max}$; we take inspiration from
	 the arguments performed in \cite[Section 3]{FaZiZh}, \cite[Section 4]{BasFeiMiz} although here we consider a very general thermodynamical setting.
	
	We point out that, since the pressure $p$ is a $C^2$-function of $(\vr,\vt)$, it is in particular bounded on bounded sets of $\mathbb{R}^2$, along with its first-order derivatives. Therefore, we deduce that
	\begin{equation} \label{boundedness thermodynamic quantitites}
		\sup_{(\vr, \vt) \in [0,\overline{\vr}] \times [0, \overline{\vt}]} \left\{ p(\vr, \vt), \ \frac{\partial p}{\partial \vr} (\vr, \vt), \ \frac{\partial p}{\partial \vt} (\vr, \vt) \right\} \leq C(\overline{\vr}, \overline{\vt}).
	\end{equation}
	
	To simplify the notation, we 
	deduce the necessary velocity and temperature estimates in terms of the \textit{material derivative}, that is defined for any given function $g$ as
	\begin{equation*}
		D_t g := \DerTime g + \vu \cdot \Grad g.
	\end{equation*}

	\subsection{Velocity estimates}
	
	\begin{lemma}[Velocity estimates, part I]
		Under the hypotheses of Theorem \ref{Conditional regularity}, there holds
		\begin{equation} \label{estimate 1}
			\begin{split}
				&\frac{\textup{d}}{\dt} \int_{\Omega} |\Grad (\vu-\vu_B)|^2 \ \dx + \int_{\Omega} \vr |D_t(\vu-\vu_B)|^2 \ \dx \\
				&\leq C(\overline{\vr}, \overline{u}, \overline{\vt}; \mathfrak{D}_0) \left(1+ \int_{\Omega} \sqrt{\vr c_{\nu}(\vr, \vt)}|D_t (\vt- \vt_B)| |\Grad(\vu-\vu_B)| \ \dx + \int_{\Omega} |\Grad (\vu-\vu_B)|^4 \ \dx   \right).
			\end{split}
		\end{equation}
	\end{lemma}
	\begin{proof}
		Proceeding as in \cite[Section 4.1]{BasFeiMiz}, we take the scalar product of the balance of momentum  \eqref{balance of momentum} with $\vu-\vu_B$ and integrate the resulting expression over $\Omega$, obtaining 
		\begin{equation} \label{e1}
			\begin{aligned}
				\int_{\Omega} &\vr |D_t (\vu-\vu_B) |^2 \ \dx - \int_{\Omega} \Div  \vS \big(\mathbb{D}_x (\vu-\vu_B)\big) \cdot D_t (\vu-\vu_B) \ \dx\\
				&=- \int_{\Omega} \Grad p \cdot D_t (\vu-\vu_B) \ \dx  + \int_{\Omega} \vr \vf \cdot D_t (\vu-\vu_B) \ \dx + \int_{\Omega} \vr D_t \vu_B \cdot D_t (\vu-\vu_B) \ \dx.
			\end{aligned}
		\end{equation}
		Notice in particular that $\Grad (\vu-\vu_B)|_{\partial \Omega} \simeq \textbf{n}$ and hence, using the fact $\vu_B$ is independent of time and tangential to the boundary, we have that
		\begin{equation} \label{boun cond v}
			D_t (\vu-\vu_B)|_{\partial \Omega} = \big[ \DerTime \vu + \vu \cdot \Grad(\vu-\vu_B)\big]|_{\partial \Omega} \simeq \vu_B \cdot \textbf{n} |_{\partial \Omega} =0.
		\end{equation}
		Noticing that 
		\begin{equation*}
			\vS \big(\mathbb{D}_x (\vu-\vu_B)\big): \Grad \DerTime (\vu-\vu_B) = \frac{1}{2} \DerTime \left[ \vS \big(\mathbb{D}_x (\vu-\vu_B)\big): \mathbb{D}_x (\vu-\vu_B) \right],
		\end{equation*}
		we can write
		\begin{equation} \label{e2}
			\begin{aligned}
				- &\int_{\Omega} \Div  \vS \big(\mathbb{D}_x (\vu-\vu_B)\big) \cdot D_t (\vu-\vu_B) \ \dx \\
				&= \frac{1}{2} \frac{\textup{d}}{\dt} \int_{\Omega} \vS \big(\mathbb{D}_x (\vu-\vu_B)\big): \mathbb{D}_x (\vu-\vu_B) \ \dx \\ 
				& + \mu \left( \int_{\Omega} \Grad (\vu-\vu_B) : \big[\Grad (\vu-\vu_B) \Grad \vu\big] \ \dx - \frac{1}{2} \int_{\Omega} |\Grad (\vu-\vu_B)|^2 \Div \vu \ \dx \right)\\
				& + \left(\eta + \frac{\mu}{3}\right)  \left(\int_{\Omega} \Div (\vu-\vu_B) \ \Grad \vu : \Grad^{\top} (\vu-\vu_B) \ \dx - \frac{1}{2}  \int_{\Omega} \big[\Div (\vu-\vu_B)\big]^2 \Div \vu \ \dx \right)
			\end{aligned}
		\end{equation}
		Furthermore, 
		\begin{equation*}
			-\int_{\Omega} \Grad p \cdot D_t (\vu-\vu_B) \ \dx = \int_{\Omega} p \ \Div D_t (\vu-\vu_B) \ \dx 
		\end{equation*}
		with 
		\begin{align*}
			p \  \Div D_t (\vu-\vu_B) = \DerTime \big[p \Div (\vu-\vu_B)\big] &- \big[ \DerTime p + \Div (p \vu)\big] \Div (\vu-\vu_B) \\
			&+p \ \Grad \vu : \Grad^{\top} (\vu-\vu_B) + \Div \big[p\vu \ \Div(\vu-\vu_B)\big];
		\end{align*}
		from the continuity equation \eqref{continuity equation} and relation \eqref{pressure and internal energy}, we can deduce that 
		\begin{equation*}
			\DerTime p + \Div (p\vu) \geq c_1 \vr c_{\nu } D_t \vt + \left(p- \vr \frac{\partial p}{\partial \vr }\right) \Div \vu.
		\end{equation*}
		Consequently,
		\begin{equation} \label{e3}
			\begin{aligned}
				-\int_{\Omega} &\Grad p \cdot D_t (\vu-\vu_B) \ \dx \\
				&\leq \frac{\textup{d}}{\dt} \int_{\Omega} p \ \Div (\vu-\vu_B) \ \dx - c_1\int_{\Omega} \vr c_{\nu}\  D_t \vt \  \Div (\vu-\vu_B) \ \dx \\
				&- \int_{\Omega} \left(p- \vr \frac{\partial p}{\partial \vr }\right) \Div \vu \  \Div (\vu-\vu_B) \ \dx + \int_{\Omega} p \  \Grad \vu : \Grad^{\top} (\vu-\vu_B) \ \dx.
			\end{aligned}
		\end{equation}
		Substituting \eqref{e2}, \eqref{e3} into \eqref{e1}, from \eqref{boundedness thermodynamic quantitites} and Young's inequality we obtain that  the following integral inequality
		\begin{equation*}
			\begin{aligned}
				\frac{1}{2} \frac{\textup{d}}{\dt} &\int_{\Omega} \vS \big(\mathbb{D}_x (\vu-\vu_B)\big): \mathbb{D}_x (\vu-\vu_B)  \ \dx + \int_{\Omega} \vr |D_t (\vu-\vu_B)|^2 \ \dx \\
				&\leq \ve_1 \frac{\textup{d}}{\dt} \int_{\Omega} |\Grad (\vu-\vu_B)|^2 \ \dx + \ve_2 \int_{\Omega} \vr |D_t (\vu-\vu_B)|^2 \ \dx \\
				&+ C(\overline{\vr}) \left( 1+ \int_{\Omega}\sqrt{\vr c_{\nu}(\vr, \vt)}  |D_t \vt|  |\Grad (\vu-\vu_B)| \ \dx + \int_{\Omega} |\Grad (\vu-\vu_B)|^4 \ \dx  \right) \\
				&+ C(\ve_1, \ve_2; \overline{\vr}, \overline{u}, \overline{\vt}) \left( 1+ \| \vu_B\|_{W^{1,2}(\Omega; \mathbb{R}^3)}^2+ \| \vf \|_{W^{1,2}(\Omega; \mathbb{R}^3)}^2  \right)
			\end{aligned}
		\end{equation*}
		holds for any $\ve_1, \ve_2>0$. Choosing $\ve_1, \ve_2$ small enough so that the first two terms on the left-hand side are absorbed by the right-hand side, we obtain \eqref{estimate 1}.
	\end{proof}

	\begin{lemma} [Velocity estimates, part II]
		Under the hypotheses of Theorem \ref{Conditional regularity}, there holds
		\begin{equation} \label{estimate 2}
			\begin{aligned}
				\frac{\textup{d}}{\dt} &\int_{\Omega} \vr |D_t (\vu-\vu_B)|^2 \ \dx + \int_{\Omega} | \Grad D_t (\vu-\vu_B)|^2 \ \dx  \\
				&\leq C(\overline{\vr}, \overline{u}, \overline{\vt}; \mathfrak{D}_0) \left(1+ \int_{\Omega} \vr c_{\nu}(\vr, \vt)|D_t (\vt- \vt_B)|^2 \ \dx + \int_{\Omega} |\Grad (\vu-\vu_B)|^4 \ \dx   \right).
			\end{aligned}
		\end{equation}
	\end{lemma}
	
	\begin{proof}
		Similarly to \cite[Section 4.2]{BasFeiMiz}, we apply the material derivative to the balance of momentum \eqref{balance of momentum}, take the scalar product of the resulting expression with $D_t(\vu-\vu_B)$ and integrate the resulting expression over $\Omega$, obtaining 
		\begin{equation} \label{e4}
			\begin{aligned}
				\int_{\Omega} &\vr D_t^2 (\vu-\vu_B) \cdot D_t (\vu-\vu_B) \ \dx + \int_{\Omega} \left[  \Grad \DerTime p + \Div ( \Grad p \otimes \vu) \right] \cdot D_t (\vu-\vu_B) \ \dx \\
				&= \mu \int_{\Omega} \left[\Delta_x \DerTime (\vu-\vu_B) + \Div \big(\Delta_x (\vu-\vu_B) \otimes \vu\big)\right] \cdot D_t (\vu-\vu_B) \ \dx \\
				&+ \left( \eta +\frac{\mu}{3} \right)  \int_{\Omega} \left[\Grad \Div \DerTime (\vu-\vu_B) + \Div \big(\Grad \Div (\vu-\vu_B) \otimes \vu \big)\right]  \dx \\
				&+ \int_{\Omega} \vr \vu \cdot \Grad \vf \cdot D_t (\vu-\vu_B)\ \dx - \int_{\Omega} \vr D_t^2 \vu_B \cdot D_t (\vu-\vu_B) \ \dx
			\end{aligned}
		\end{equation}
		Notice that 
		\begin{equation} \label{e5}
			\vr D_t^2 (\vu-\vu_B) \cdot D_t (\vu-\vu_B)= \frac{1}{2} \left[\DerTime\big(\vr |D_t (\vu-\vu_B)|^2\big) + \Div \big(\vr |D_t (\vu-\vu_B)|^2 \vu  \big)\right]. 
		\end{equation}
		while, from \eqref{boun cond v} we additionally deduce that
		\begin{equation} 
			\begin{aligned}
				\int_{\Omega} &\left[  \Grad \DerTime p + \Div ( \Grad p \otimes \vu) \right] \cdot D_t (\vu-\vu_B) \ \dx \\
				&= - \int_{\Omega} \left[ \DerTime p + \Div (p\vu)\right] \Div D_t (\vu-\vu_B) \ \dx + \int_{\Omega} p \Grad \vu \cdot \Grad D_t (\vu-\vu_B) \ \dx.
			\end{aligned}
		\end{equation}
		Moreover,
		\begin{equation*}
			\begin{aligned}
				\int_{\Omega} &\left[\Delta_x \DerTime (\vu-\vu_B) + \Div \big(\Delta_x (\vu-\vu_B) \otimes \vu\big)\right]  \cdot D_t (\vu-\vu_B) \ \dx \\
				=& -\int_{\Omega} \left[ \Grad \DerTime (\vu-\vu_B) + \big(\Delta_x (\vu-\vu_B) \otimes \vu\big)\right] : \Grad D_t (\vu-\vu_B) \ \dx \\
				=& - \int_{\Omega} |\Grad D_t (\vu-\vu_B) |^2 \ \dx \\
				&+ \int_{\Omega} \left[ \Grad \big(\vu \cdot \Grad (\vu-\vu_B)\big)- (\Delta_x (\vu-\vu_B) \otimes \vu) \right] : \Grad D_t (\vu-\vu_B) \ \dx
			\end{aligned}
		\end{equation*}
		using the summation convention and letting $\vv= \vu-\vu_B$, $\vw= D_t (\vu-\vu_B)$, we have
		\begin{align*}
			\int_{\Omega} &\Grad (\vu \cdot \Grad \vv): \Grad \vw \ \dx \\
			&= \int_{\Omega} (\Grad \vv \Grad \vu) : \Grad \vw \ \dx + \int_{\Omega} u_j \partial_{x_j}\partial_{x_k} v_i \partial_{x_k} w_i \ \dx \\
			&= \int_{\Omega} (\Grad \vv \Grad \vu) : \Grad \vw \ \dx  - \int_{\Omega} (\Div \vu) \Grad \vv : \Grad \vw \ \dx + \int_{\Omega} \partial_{x_j} \big(u_j \partial_{x_k} v_i \big) \partial_{x_k} w_i \ \dx \\
			&= \int_{\Omega} (\Grad \vv \Grad \vu) : \Grad \vw \ \dx  - \int_{\Omega} (\Div \vu) \Grad \vv : \Grad \vw \ \dx + \int_{\Omega} \partial_{x_k} \big(u_j \partial_{x_k} v_i \big) \partial_{x_j} w_i \ \dx \\
			&= \int_{\Omega} (\Grad \vv \Grad \vu) : \Grad \vw \ \dx  - \int_{\Omega} (\Div \vu) \Grad \vv : \Grad \vw \ \dx \\
			&+ \int_{\Omega} (\Delta_x \vv \otimes \vu ): \Grad \vw \ \dx + \int_{\Omega} \Grad \vv : (\Grad \vw \Grad \vu ) \ \dx,
		\end{align*}
		where in the fourth line we have performed an integration by parts. Therefore, we obtain 
		\begin{equation}
			\begin{aligned}
				\int_{\Omega} &\left[\Delta_x \DerTime (\vu-\vu_B) + \Div \big(\Delta_x (\vu-\vu_B) \otimes \vu\big)\right]  \cdot D_t (\vu-\vu_B) \ \dx \\
				= & - \int_{\Omega} |\Grad D_t (\vu-\vu_B) |^2 \ \dx - \int_{\Omega} (\Div \vu) \Grad (\vu-\vu_B) : \Grad D_t(\vu-\vu_B) \ \dx \\
				& + \int_{\Omega} \Big[(\Grad (\vu-\vu_B) \Grad \vu) : \Grad D_t(\vu-\vu_B)  +  \Grad (\vu-\vu_B) : \big(\Grad D_t(\vu-\vu_B) \Grad \vu \big)\Big] \  \dx.
			\end{aligned}
		\end{equation}
		Similarly,
		\begin{align*}
			\int_{\Omega} &\left[\Grad \Div \DerTime (\vu-\vu_B)+ \Div \big(\Grad \Div (\vu-\vu_B) \otimes \vu \big)\right]  \cdot D_t (\vu-\vu_B) \ \dx \\
			=& - \int_{\Omega} |\Div D_t (\vu-\vu_B)|^2 \ \dx \\
			&+ \int_{\Omega} \left[ \Div \big(\vu \cdot \Grad (\vu-\vu_B)\big) \ \Div D_t(\vu-\vu_B) - \big(\Grad \Div (\vu-\vu_B) \otimes \vu \big) : \Grad D_t(\vu-\vu_B) \right] \dx;
		\end{align*}
		once again, using the summation convention and letting $\vv= \vu-\vu_B$, $\vw= D_t (\vu-\vu_B)$, we obtain 
		\begin{align*}
			&\int_{\Omega}  \Div (\vu \cdot \Grad \vv) \ \Div \vw \ \dx \\
			&= \int_{\Omega} (\Grad \vv : \Grad^{\top} \vu) \ \Div \vw \ \dx + \int_{\Omega} u_k \partial_{x_k} \partial_{x_j} v_j \partial_{x_i} w_i \ \dx \\
			&= \int_{\Omega} (\Grad \vv : \Grad^{\top} \vu) \ \Div \vw \ \dx- \int_{\Omega} (\Div \vv) \ (\Div \vu) \ (\Div \vw) \ \dx + \int_{\Omega} \partial_{x_k} \big(  (\Div \vv) u_k \big) \partial_{x_i} w_i \ \dx \\
			&= \int_{\Omega} (\Grad \vv : \Grad^{\top} \vu) \ \Div \vw \ \dx- \int_{\Omega} (\Div \vv) \ (\Div \vu) \ (\Div \vw) \ \dx + \int_{\Omega} \partial_{x_i} \big(  (\Div \vv) u_k \big) \partial_{x_k} w_i \ \dx \\
			&= \int_{\Omega} (\Grad \vv : \Grad^{\top} \vu) \ \Div \vw \ \dx- \int_{\Omega} (\Div \vv) \ (\Div \vu) \ (\Div \vw) \ \dx \\
			&+ \int_{\Omega} (\Grad \Div \vv \otimes \vu ) : \Grad \vw \ \dx + \int_{\Omega} (\Div \vv) \Grad \vu : \Grad^{\top} \vw \ \dx.
		\end{align*}
		Thus, we conclude that 
		\begin{equation}
			\begin{aligned}
				\int_{\Omega} &\left[\Grad \Div \DerTime (\vu-\vu_B)+ \Div \big(\Grad \Div (\vu-\vu_B) \otimes \vu \big)\right]  \cdot D_t (\vu-\vu_B) \ \dx \\
				=& - \int_{\Omega} |\Div D_t (\vu-\vu_B)|^2 \ \dx - \int_{\Omega} (\Div \vu) \  \big(\Div (\vu-\vu_B) \big) \ \big(\Div D_t(\vu-\vu_B)\big ) \ \dx \\
				& + \int_{\Omega} \Big[\big(\Grad (\vu-\vu_B) : \Grad^{\top} \vu\big) \ \Div D_t(\vu-\vu_B) + \big(\Div (\vu-\vu_B)\big ) \Grad \vu : \Grad^{\top} D_t(\vu-\vu_B)\Big] \  \dx. 
			\end{aligned}
		\end{equation}
		Finally, notice that 
		\begin{equation} \label{e6}
				D_t^2 \vu_B = D_t (\vu-\vu_B) \cdot \Grad \vu_B + (\vu \otimes \vu): \Grad^2 \vu_B + (\vu \cdot \Grad \vu_B) \cdot \Grad \vu_B.
		\end{equation}
		Substituting \eqref{e5}--\eqref{e6} into \eqref{e4}, from \eqref{boundedness thermodynamic quantitites} and Young's inequality, we obtain that the following integral inequality
		\begin{equation*}
			\begin{aligned}
				\frac{1}{2} &\frac{\textup{d}}{\dt} \int_{\Omega} \vr |D_t (\vu-\vu_B)|^2 \ \dx + \mu \int_{\Omega} | \Grad D_t (\vu-\vu_B)|^2 \ \dx + \left( \eta + \frac{\mu}{3}\right) \int_{\Omega} |\Div D_t (\vu-\vu_B)|^2 \ \dx \\
				& \leq \ve_1 \int_{\Omega} | \Grad D_t (\vu-\vu_B)|^2 \ \dx + \ve_2 \int_{\Omega} |\Div D_t (\vu-\vu_B)|^2 \ \dx \\
				&+ \left(1+ \| \Grad \vu_B\|_{L^{\infty}(\Omega; \mathbb{R}^{3\times 3})} \right) \int_{\Omega} \vr |D_t (\vu-\vu_B)|^2 \ \dx \\
				&+ C(\ve_1, \ve_2; \overline{\vr}, \overline{\vt})  \left( 1+ \int_{\Omega}\vr c_{\nu}(\vr,\vt)  |D_t \vt|^2  \ \dx + \int_{\Omega} |\Grad (\vu-\vu_B)|^4 \ \dx  \right) \\
				&+  C(\ve_1, \ve_2; \overline{\vr}, \overline{u}, \overline{\vt}) \left( \| \Grad \vu_B\|_{L^4(\Omega; \mathbb{R}^{3\times 3 })}^4 + \| \Grad^2 \vu_B\|_{L^2(\Omega; \mathbb{R}^{3\times 3 \times 3})}^2 + \| \Grad \vf \|_{L^2(\Omega; \mathbb{R}^{3\times 3 })}^2  \right)
			\end{aligned}
		\end{equation*}
		holds for any $\ve_1, \ve_2>0$. Again, choosing $\ve_1, \ve_2$ small enough and using \eqref{estimate 1}, we obtain \eqref{estimate 2}.
	\end{proof}

	\subsection{Temperature estimates}
	
	\begin{lemma}[Temperature estimates]
		Under the hypotheses of Theorem \ref{Conditional regularity}, there holds
		\begin{equation} \label{estimate 3}
			\begin{aligned}
				\frac{\textup{d}}{\dt} &\int_{\Omega} |\Grad (\vt-\vt_B)|^2 \ \dx + \int_{\Omega} \vr c_{\nu}(\vr, \vt) |D_t (\vt-\vt_B)|^2 \ \dx  \\
				&\leq C(\overline{\vr}, \overline{u}, \overline{\vt}; \mathfrak{D}_0) \left(1+ \int_{\Omega} |\Grad  (\vt- \vt_B)|^2 \ \dx + \int_{\Omega} |\Grad (\vu-\vu_B)|^4 \ \dx   \right).
			\end{aligned}
		\end{equation}
	\end{lemma}
	\begin{proof}
		From the continuity equation \eqref{continuity equation}, \eqref{extension vt_B} and Maxwell relation \ref{Maxwell relation}, the balance of internal energy \eqref{balance internal energy} can be rewritten as 
		\begin{equation} \label{balance of internal energy rev }
			\vr c_{\nu}(\vr, \vt) D_t (\vt-\vt_B) - \kappa \Delta_x (\vt-\vt_B) + \vt \frac{\partial p}{\partial \vt} \Div \vu = \vS (\mathbb{D}_x \vu) : \mathbb{D}_x \vu - \vr c_{\nu}(\vr, \vt) \vu \cdot \Grad \vt_B.
		\end{equation}
		Following \cite[Section 4.4]{BasFeiMiz}, we multiply the previous identity by $\DerTime (\vt-\vt_B)$ and integrate the resulting expression over $\Omega$ to get
		\begin{equation} \label{e9}
			\begin{aligned}
				&\frac{\kappa}{2} \frac{\textup{d}}{\dt} \int_{\Omega} |\Grad (\vt-\vt_B)|^2 \ \dx + \int_{\Omega} \vr c_{\nu} |D_t (\vt-\vt_B)|^2 \ \dx \\
				&= \int_{\Omega} \vr c_{\nu} D_t (\vt-\vt_B) \ \big(\vu \cdot \Grad (\vt-\vt_B) \big) \ \dx - \int_{\Omega} \vt \frac{\partial p}{\partial \vt} D_t (\vt-\vt_B) \ \Div \vu \ \dx \\
				&+ \int_{\Omega} \vt \frac{\partial p}{\partial \vt} \ \Div \vu \ \big(\vu \cdot \Grad (\vt-\vt_B) \big) \ \dx - \int_{\Omega} \vr c_{\nu} D_t (\vt-\vt_B) \ (\vu \cdot \Grad \vt_B) \ \dx \\
				&+ \int_{\Omega} \vr c_{\nu} (\vu \cdot \Grad \vt_B) \ \big(\vu \cdot \Grad (\vt-\vt_B) \big) \ \dx + \frac{\textup{d}}{\dt} \int_{\Omega} \vt \  \vS(\mathbb{D}_x \vu): \mathbb{D}_x \vu \ \dx \\
				&- 2\mu \int_{\Omega} \vt \ \mathbb{D}_x \vu : \Grad \DerTime \vu \ \dx - 2\left( \eta - \frac{2}{3} \mu \right) \int_{\Omega} \vt \  \Div \vu \ \Div \DerTime \vu \ \dx.
			\end{aligned}
		\end{equation}
		We will now focus on the last two integrals. We have 
		\begin{equation} \label{e7}
			\begin{aligned}
				&\int_{\Omega} \vt \ \mathbb{D}_x\vu : \Grad \DerTime \vu \ \dx \\
				&= \int_{\Omega} \vt \  \mathbb{D}_x \vu : \Grad D_t (\vu-\vu_B) \ \dx \\
				&- \int_{\Omega} \vt \ \mathbb{D}_x (\vu-\vu_B) : \Grad \big(\vu \cdot \Grad (\vu-\vu_B)\big) \ \dx - \int_{\Omega} \vt \ \mathbb{D}_x \vu_B : \Grad \big(\vu \cdot \Grad (\vu-\vu_B)\big) \ \dx \\
				&= \int_{\Omega} \vt \  \mathbb{D}_x \vu : \Grad D_t (\vu-\vu_B) \ \dx \\
				& -\int_{\Omega} \vt \  \mathbb{D}_x (\vu-\vu_B) : \big(\Grad (\vu-\vu_B) \Grad \vu\big) \ \dx + \frac{1}{2} \int_{\Omega} \big(\mathbb{D}_x (\vu-\vu_B) : \Grad (\vu-\vu_B)\big) \ \Div (\vt \vu ) \  \dx \\
				&- \int_{\Omega} \vt \  \mathbb{D}_x \vu_B : \big(\Grad (\vu-\vu_B) \Grad \vu\big) \ \dx + \int_{\Omega} \big(\mathbb{D}_x  \vu_B : \Grad (\vu-\vu_B)\big)\ \Div (\vt \vu) \ \dx \\
				&+ \int_{\Omega} \vt \  \mathbb{D}_x  (\vu-\vu_B) : (\vu \cdot \Grad^2 \vu_B) \ \dx.
			\end{aligned}
		\end{equation}
		Similarly, 
		\begin{equation} \label{e8}
			\begin{aligned}
				&\int_{\Omega} \vt \  \Div \vu \ \Div \DerTime \vu \ \dx \\
				&= \int_{\Omega} \vt \  \Div \vu \ \Div D_t (\vu-\vu_B) \ \dx \\
				&- \int_{\Omega} \vt \  \Div (\vu-\vu_B) \  \Div \big(\vu\cdot \Grad (\vu-\vu_B)\big) \ \dx - \int_{\Omega} \vt \  \Div \vu_B \  \Div \big(\vu\cdot \Grad (\vu-\vu_B)\big) \ \dx \\
				&= \int_{\Omega} \vt \  \Div \vu \ \Div D_t (\vu-\vu_B) \ \dx \\
				&- \int_{\Omega} \vt \ \Div (\vu-\vu_B) \ \big(\Grad \vu: \Grad^{\top} (\vu-\vu_B)\big) \ \dx + \frac{1}{2} \int_{\Omega} \big(\Div (\vu-\vu_B)\big)^2 \ \Div(\vt \vu) \ \dx \\
				&- \int_{\Omega} \vt \ \Div \vu_B \ \big(\Grad \vu: \Grad^{\top} (\vu-\vu_B)\big) \ \dx + \int_{\Omega} (\Div \vu_B) \ \big(\Div (\vu-\vu_B) \big) \ \Div (\vt \vu) \ \dx \\
				&+ \int_{\Omega} \vt \ \big(\Div (\vu-\vu_B) \big) \ \vu \cdot \Grad \Div \vu_B \ \dx.
			\end{aligned}
		\end{equation}
		As before, substituting \eqref{e7} and \eqref{e8} into \eqref{e9}, we obtain that  the following integral inequality
		\begin{equation*}
			\begin{aligned}
				\frac{\kappa}{2} &\frac{\textup{d}}{\dt} \int_{\Omega} |\Grad (\vt-\vt_B)|^2 \ \dx + \int_{\Omega} \vr c_{\nu}(\vr, \vt) |D_t (\vt-\vt_B)|^2 \ \dx \\
				&\leq \ve_1 \int_{\Omega} \vr c_{\nu}(\vr, \vt) |D_t (\vt-\vt_B)|^2 \ \dx + \ve_2 \int_{\Omega} |\Grad D_t (\vu-\vu_B) |^2 \ \dx \\
				&+ C(\overline{\vt}) \frac{\textup{d}}{\dt} \int_{\Omega} |\Grad (\vu-\vu_B)|^2 \ \dx \\
				&+ C(\ve_1, \ve_2; \overline{\vr}, \overline{u}, \overline{\vt})  \left( 1+ \int_{\Omega} |\Grad (\vt-\vt_B)|^2  \ \dx + \int_{\Omega} |\Grad (\vu-\vu_B)|^4 \ \dx  \right) \\
				&+  C(\ve_1, \ve_2; \overline{\vr}, \overline{u}, \overline{\vt}) \left( \| \Grad \vu_B\|_{L^4(\Omega; \mathbb{R}^{3\times 3 })}^4 + \| \Grad^2 \vu_B\|_{L^2(\Omega; \mathbb{R}^{3\times 3 \times 3})}^2 + \| \Grad \vt_B \|_{L^2(\Omega; \mathbb{R}^{3})}^2  \right)
			\end{aligned}
		\end{equation*}
		holds for any $\ve_1, \ve_2>0$.  Moreover, from \eqref{estimate 1} we have that 
		\begin{equation*}
			\begin{aligned}
				\frac{\textup{d}}{\dt} \int_{\Omega} |\Grad (\vu-\vu_B)|^2 \ \dx &\leq \ve \int_{\Omega} \vr c_{\nu}(\vr,\vt) |D_t (\vt-\vt_B)|^2 \ \dx \\
				&+ C(\ve; \overline{\vr}, \overline{u}, \overline{\vt}; \mathfrak{D}_0) \left(1+ \int_{\Omega} |\Grad (\vu-\vu_B) |^4 \ \dx   \right),
			\end{aligned}
		\end{equation*}
		holds for any $\ve>0$. Therefore, using \eqref{estimate 2} and choosing $\ve_1, \ve_2$ sufficiently small, we get \eqref{estimate 3}. 
	\end{proof}

	\subsection{Gronwall argument}
	
	Putting together estimates \eqref{estimate 1}, \eqref{estimate 2} and \eqref{estimate 3}, we finally obtain 
	\begin{equation*}
		\begin{aligned}
			\frac{\textup{d}}{\dt} &\int_{\Omega} \Big(\vr |D_t(\vu-\vu_B)|^2 + |\Grad (\vu-\vu_B)|^2+ |\Grad (\vt-\vt_B)|^2 \Big) \ \dx \\
			&+ \int_{\Omega}  \Big( \vr |D_t(\vu-\vu_B)|^2 + \vr c_{\nu}(\vr,\vt) |D_t(\vt-\vt_B)|^2 + |\Grad D_t(\vu-\vu_B)|^2 \Big) \  \dx \\
			&\leq C(\overline{\vr}, \overline{u}, \overline{\vt}; \mathfrak{D}_0) \left(1+ \int_{\Omega} |\Grad  (\vt- \vt_B)|^2 \ \dx + \int_{\Omega} |\Grad (\vu-\vu_B)|^4 \ \dx   \right).
		\end{aligned}
	\end{equation*}
	We can now integrate the previous inequality over $[0, \tau]$, $0<\tau \leq T< T_{\rm max}$ to get 
	\begin{equation} \label{final estimate 1}
		\begin{aligned}
			\int_{\Omega} &\Big( \vr |D_t(\vu-\vu_B)|^2 + |\Grad (\vu-\vu_B)|^2+ |\Grad (\vt-\vt_B)|^2 \Big) (\tau, \cdot ) \  \dx \\
			&+ \int_{0}^{\tau}\int_{\Omega}  \left(\vr |D_t(\vu-\vu_B)|^2 + \vr c_{\nu}(\vr, \vt) |D_t(\vt-\vt_B)|^2 + |\Grad D_t(\vu-\vu_B)|^2 \right) \dx \dt \\
			&\leq C(T_{\rm max}; \overline{\vr}, \overline{u}, \overline{\vt}; \mathfrak{D}_0) \left(1+\int_{0}^{\tau} \int_{\Omega} |\Grad  (\vt- \vt_B)|^2 \ \dx \dt + \int_{0}^{\tau}\int_{\Omega} |\Grad (\vu-\vu_B)|^4 \ \dx  \dt  \right).
		\end{aligned}
	\end{equation}
	In order to apply the Gronwall argument, we need to properly estimate the last term appearing on the right-hand side of \eqref{final estimate 1}. To this end, we follow the idea developed in \cite{SunWanZha1} and adapted in \cite[Section 4.3]{BasFeiMiz}: we decompose $\vu-\vu_B$ as 
	\begin{align}
		\vu-\vu_B&= \vv +\vw, \\
		\Div \vS (\mathbb{D}_ x \vv) &= \Grad p,  \\ 
		\Div \vS (\mathbb{D}_ x \vw) &= \vr D_t (\vu-\vu_B)- \vr \vf + \vr (\vu \cdot \nabla_x \vu_B)  \label{v5}
	\end{align}
	in $(0,T) \times \Omega$, with $\vv|_{\partial \Omega}= \vw|_{\partial \Omega} =0$. Conditions \eqref{boundedness} and \eqref{boundedness thermodynamic quantitites} yield
	\begin{align}
		\| \vv \|_{W^{1,q}(\Omega; \mathbb{R}^3)} &\leq c(q; \overline{\vr}, \overline{\vt}), \label{v1}\\
		\| \vv \|_{W^{2,q}(\Omega; \mathbb{R}^3)} &\leq c(q; \overline{\vr}, \overline{\vt}) \left( \| \Grad \vr \|_{L^q(\Omega)} + \| \Grad \vt \|_{L^q(\Omega)} \right), \label{v2}\\
		\| \vw \|_{W^{2,2}(\Omega; \mathbb{R}^3)} &\leq c(\overline{\vr}, \overline{u}; \mathfrak{D}_0) \left( 1+\| \sqrt{\vr} D_t (\vu-\vu_B)\|_{L^2(\Omega; \mathbb{R}^3)}\right), \label{v3}
	\end{align}
	for any $1 \leq q<\infty$. Notice, in particular, that from \eqref{boundedness} and \eqref{v1}
	\begin{equation*}
		\| \vw \|_{L^{\infty}(\Omega)} \leq \| \vu- \vu_B \|_{L^{\infty}(\Omega)} + \| \vv \|_{L^{\infty}(\Omega)} \leq C(\overline{\vr}, \overline{u}, \overline{\vt}; \mathfrak{D}_0),
	\end{equation*}
	and therefore we can use the Gagliardo-Nirenberg interpolation inequality and \eqref{e3} to deduce that 
	\begin{equation} \label{v4}
		\| \Grad \vw \|_{L^4(\Omega)}^2 \lesssim \| \vw \|_{L^{\infty}(\Omega)} \| \Delta_x \vw \|_{L^2(\Omega)} \leq C(\overline{\vr}, \overline{u}, \overline{\vt}; \mathfrak{D}_0) \left( 1+\| \sqrt{\vr} D_t (\vu-\vu_B) \|_{L^2(\Omega)} \right).
	\end{equation}
	Hence, from \eqref{v1} and \eqref{v4}, we deduce that  
	\begin{equation} \label{e15}
		\begin{aligned}
			\| \Grad (\vu-\vu_B) \|_{L^4(\Omega)}^4 &\leq \| \Grad \vv \|_{L^4(\Omega)}^4 + \| \Grad \vw \|_{L^4(\Omega)}^4 \\
			&\leq C(\overline{\vr}, \overline{u}, \overline{\vt}; \mathfrak{D}_0) \left( 1+\int_{\Omega}\vr |D_t (\vu-\vu_B)|^2 \ \dx \right). 
		\end{aligned}
	\end{equation}
	Substituting \eqref{e15} into \eqref{final estimate 1}, we obtain 
	\begin{equation} \label{final estimate 2}
		\begin{aligned}
			\int_{\Omega} &\big(\vr |D_t(\vu-\vu_B)|^2 + |\Grad (\vu-\vu_B)|^2+ |\Grad (\vt-\vt_B)|^2 \big) (\tau, \cdot ) \  \dx \\
			&+ \int_{0}^{\tau}\int_{\Omega}  \Big(\vr |D_t(\vu-\vu_B)|^2 + \vr c_{\nu}(\vr, \vt) |D_t(\vt-\vt_B)|^2 + |\Grad D_t(\vu-\vu_B)|^2 \Big) \dx \dt \\
			&\leq C(T_{\rm max }; \overline{\vr}, \overline{u}, \overline{\vt}; \mathfrak{D}_0) \left(1+ \int_{0}^{\tau}\int_{\Omega}  \vr |D_t(\vu-\vu_B)|^2 \ \dx\dt + \int_{0}^{\tau} \int_{\Omega} |\Grad  (\vt- \vt_B)|^2 \ \dx \dt  \right).
		\end{aligned}
	\end{equation}
	By a standard Gronwall argument we can deduce the following bounds:
	\begin{align}
		\sup_{t\in [0,T]} \| \sqrt{\vr} D_t (\vu -\vu_B)(t, \cdot)\|_{L^2(\Omega; \mathbb{R}^3)} & \leq C(T_{\rm max}; \overline{\vr}, \overline{u}, \overline{\vt}; \mathfrak{D}_0), \label{e10}\\
		\sup_{t\in [0,T]} \| (\vu-\vu_B) (t, \cdot)\|_{W^{1,2}_0(\Omega; \mathbb{R}^3)} & \leq C(T_{\rm max}; \overline{\vr}, \overline{u}, \overline{\vt}; \mathfrak{D}_0), \label{e11}\\
		\sup_{t\in [0,T]} \| (\vt-\vt_B) (t, \cdot)\|_{W^{1,2}_0(\Omega)} & \leq C(T_{\rm max}; \overline{\vr}, \overline{u}, \overline{\vt}; \mathfrak{D}_0), \label{e12}\\
		\int_{0}^{T} \int_{\Omega} \vr c_{\nu}(\vr, \vt) |D_t (\vt-\vt_B)|^2 \ \dx \dt &\leq C(T_{\rm max}; \overline{\vr}, \overline{u}, \overline{\vt}; \mathfrak{D}_0), \label{e13}\\
		\int_{0}^{T} \int_{\Omega} |\Grad D_t (\vu-\vu_B)|^2 \ \dx \dt &\leq C(T_{\rm max}; \overline{\vr}, \overline{u}, \overline{\vt}; \mathfrak{D}_0). \label{e14}
	\end{align}
	
	\section{Positivity of the density and the temperature}\label{posit}
	
	In this section. our goal is to obtain the bounds from below for the density and temperature, i.e. to prove estimate \eqref{bound from below density and temperature}.
	
	\subsection{Higher-order estimates}
	
	Applying the standard parabolic estimates to the internal energy balance \eqref{balance of internal energy rev }, from \eqref{e15}, \eqref{e11}, \eqref{e13} and \eqref{e14} we deduce 
	\begin{equation} \label{e16}
		\int_{0}^{T} \| (\vt-\vt_B)(t,\cdot)\|_{W^{2,2}(\Omega)}^2 \ \dt \leq C(T_{\max}; \overline{\vr}, \overline{u}, \overline{\vt}; \mathfrak{D}_0).
	\end{equation}
	Therefore, we may proceed as in \cite[Section 5]{SunWanZha1} to deduce from \eqref{v5}, \eqref{e14}, \eqref{e16} and the Sobolev embedding $W^{1,2}(\Omega) \hookrightarrow L^6(\Omega)$ that 
	\begin{equation} \label{e17}
		\sup_{t\in [0,T]} \| \Grad \vr (t,\cdot) \|_{L^6(\Omega; \mathbb{R}^3)} \leq C(T_{\max}; \overline{\vr}, \overline{u}, \overline{\vt}; \mathfrak{D}_0).
	\end{equation}
	Adapting the Lam\'{e} estimates to the balance of momentum \eqref{balance of momentum}, we finally obtain from \eqref{e14}, \eqref{e16}, \eqref{e17} and the Sobolev embedding $W^{1,2}(\Omega) \hookrightarrow L^6(\Omega)$ that
	\begin{equation} \label{e18}
		\int_{0}^{T} \| (\vu-\vu_B)(t, \cdot)\|_{W^{2,6}(\Omega; \mathbb{R}^3)}^2 \ \dt \leq C(T_{\max}; \overline{\vr}, \overline{u}, \overline{\vt}; \mathfrak{D}_0).
	\end{equation}

	\subsection{Lower bound for the density}
	
	Following \cite[Section 6.1]{BasFeiMiz}, from \eqref{e18} and the Sobolev embedding $W^{1,6}(\Omega) \hookrightarrow L^{\infty}(\Omega)$, we deduce that $\Div \vu \in L^1(0,T; L^{\infty}(\Omega))$. Consequently, from the continuity equation \eqref{continuity equation} and, in particular, from the fact that 
	\begin{equation*}
		\inf_{\overline{\Omega}}\vr_0 \ \exp \left(-\int_{0}^{T} \| \Div \vu \|_{L^{\infty}(\Omega)} \ \dt \right) \leq \vr(t, x )  \leq \sup_{\overline{\Omega}}\vr_0 \ \exp \left(\int_{0}^{T} \| \Div \vu \|_{L^{\infty}(\Omega)} \ \dt \right),
	\end{equation*}
	we recover a positive lower bound on the density, 
	\begin{equation} \label{positivity density}
		\inf_{(t,x) \in [0,T] \times \overline{\Omega}} \vr (t,x) \geq \underline{\vr} >0,
	\end{equation}
	with $\underline{\vr}= \underline{\vr}(T_{\rm max}; \overline{\vr}, \overline{u}, \overline{\vt}; \mathfrak{D}_0) $. 
	
	\subsection{Lower bound for the temperature}
	
	First of all, from the strict positivity of the bulk viscosity and relation \eqref{pressure and internal energy}, we can write 
	\begin{equation} \label{e19}
 \vt \frac{\partial p}{\partial \vt} \Div \vu \geq -  \frac{\eta}{2} |\Div \vu|^2 -  C(\overline{\vr}, \overline{\vt}; \mathfrak{D}_0)  \vt \vr c_{\nu}(\vr, \vt) .
	\end{equation} 
	Moreover, we point out that the temperature $\vt$ constructed in Theorem \ref{Local existence} is strictly positive; therefore, due to thermodynamic stability \eqref{thermodynamic stability}, we can divide the balance of internal energy \eqref{balance of internal energy rev } by $\vr c_{\nu}(\vr, \vt)$ and use \eqref{e19} to get
	\begin{equation*}
		\DerTime \vt + \vu \cdot \Grad \vt - \frac{\kappa}{\vr c_{\nu}(\vr, \vt)} \Delta_x \vt +  c_0 \vt \geq 0.
	\end{equation*}
	 with $c_0=c_0(\overline{\vr}, \overline{\vt}; \mathfrak{D}_0)$.
	Now, it's enough to apply the standard minimum principle for parabolic equations to get the for any $\lambda >c_0$, 
	\begin{equation} \label{positivity temperature}
		\inf_{(t,x) \in [0,T] \times \overline{\Omega}} \vt (t,x) \geq \underline{\vt}:=  e^{-\lambda T} \min \{ \underline{\vt}_0, \underline{\vt}_B \};
	\end{equation}
	in particular, if $T_{\max}$ is finite, we obtain that $\underline{\vt}=  \underline{\vt}(T_{\rm max}; \overline{\vr}, \overline{u}, \overline{\vt}; \mathfrak{D}_0) $.
	
	\section{Final estimate}\label{final}
	
	Having obtained \eqref{bound from below density and temperature}, we can now proceed as in \cite[Sections 7 and 8]{BasFeiMiz} to get the final estimates \eqref{final estimate}. 
	
	\subsection{H\"{o}lder-continuity} 
	
		We begin pointing out that, putting together \eqref{v1}, \eqref{v3} and \eqref{e10}, for any $0< T< T_{\rm max}$,
		\begin{equation} \label{e20}
			\sup_{t\in [0,T]} \| \Grad (\vu-\vu_B)(t, \cdot) \|_{L^6(\Omega; \mathbb{R}^{3\times 3})}  \leq  C(T_{\rm max}; \overline{\vr}, \overline{u}, \overline{\vt}; \mathfrak{D}_0).
		\end{equation}
		Therefore, from the continuity equation \eqref{continuity equation} and \eqref{e17},
		\begin{equation*}
			\sup_{t\in [0,T]} \left(\| \DerTime \vr(t, \cdot)\|_{L^6(\Omega)} + \| \vr (t, \cdot) \|_{W^{1,6}(\Omega)} \right) \leq C(T_{\rm max}; \overline{\vr}, \overline{u}, \overline{\vt}; \mathfrak{D}_0);
		\end{equation*}
		by interpolation, we deduce that $\vr$ is H\"{o}lder-continuous on $[0,T] \times \overline{\Omega}$. Moreover, from \eqref{positivity density}, \eqref{positivity temperature} and the termodynamic stability \eqref{thermodynamic stability}, we can divide the balance of internal energy \eqref{balance of internal energy rev } by $\vr c_{\nu} (\vr,\vt)$ and apply the standard parabolic theory to get that $\vt$ is H\"{o}lder-continuous on $[0,T] \times \overline{\Omega}$. Summarizing, for some $\alpha>0$, we have that 
		\begin{equation} \label{Holder continuity}
			\| (\vr, \vt) \|_{C^{\alpha}([0,T] \times \overline{\Omega}; \mathbb{R}^2)} \leq C(T_{\rm max}; \overline{\vr}, \overline{u}, \overline{\vt}; \mathfrak{D}_0).
		\end{equation}
		\subsection{Parabolic regularity} 
		
		Having established the H\"{o}lder-continuity for  $\vr$ and $\vt$, we can now apply the maximal $L^p$-$L^q$ regularity to equation
		\begin{equation} \label{bal int ene}
			\DerTime \vt + \vu \cdot \Grad \vt - \frac{\kappa}{\vr c_{\nu}(\vr, \vt)} \Delta_x \vt= \frac{1}{\vr c_{\nu}(\vr, \vt)} \left(\vS (\mathbb{D}_x \vu): \mathbb{D}_x \vu - \vt \frac{\partial p}{\partial \vt} \Div \vu  \right).
		\end{equation}
		More precisely, noticing that from \eqref{positivity density}, \eqref{positivity temperature} and \eqref{Holder continuity},
		\begin{equation*}
			\sup_{(t,x) \in [0,T] \times \overline{\Omega}} \frac{1}{\big(\vr c_{\nu}(\vr, \vt)\big)(t,x)} \leq C(T_{\rm max}; \overline{\vr}, \overline{u}, \overline{\vt}; \mathfrak{D}_0),
		\end{equation*}
		and that, from \eqref{boundedness thermodynamic quantitites}, \eqref{e20}, the right-hand side of \eqref{bal int ene} is bounded in $L^{\infty}(0,T; L^3(\Omega))$ by a constant independent of $0<T< T_{\rm max}$, we can apply \cite[Theorem 2.3]{DenHeiPru} to get that the following estimate
		\begin{equation} \label{e21}
			\| \DerTime \vt \|_{L^p(0,T; L^3(\Omega))}  + \| \vt \|_{L^p(0,T; W^{2,3}(\Omega))} \leq C(T_{\rm max}; \overline{\vr}, \overline{u}, \overline{\vt}; \mathfrak{D}_0)
		\end{equation}
		holds for any $1<p<\infty$. A similar argument can be repeated for the balance of momentum \eqref{balance of momentum},
		\begin{equation} \label{bal mom}
			\DerTime \vu + \vu \cdot \Grad \vu - \frac{1}{\vr} \left[\mu \Delta_x \vu + \left( \eta + \frac{\mu}{3}\right) \nabla_x \Div \vu \right] = -\frac{1}{\vr} \left( \frac{\partial p }{\partial \vr} \Grad \vr + \frac{\partial p}{\partial \vt} \Grad \vt \right) + \vf
		\end{equation}
		where now we use \eqref{e17} and \eqref{e21} to deduce that the right-hand side of \eqref{bal mom} is bounded in $L^{\infty}(0,T; L^6(\Omega))$ by a constant independent of $0<T< T_{\rm max}$. Thus, we obtain that the following estimate
		\begin{equation} \label{e22}
			\| \DerTime \vu \|_{L^p(0,T; L^6(\Omega))}  + \| \vu \|_{L^p(0,T; W^{2,6}(\Omega))} \leq C(T_{\rm max}; \overline{\vr}, \overline{u}, \overline{\vt}; \mathfrak{D}_0)
		\end{equation}
		holds for any $1<p<\infty$, from which, by interpolation, it is easy to deduce that 
		\begin{equation*}
			\sup_{(t,x) \in [0,T] \times \overline{\Omega}} | \Grad \vu(t,x) | \leq C(T_{\rm max}; \overline{\vr}, \overline{u}, \overline{\vt}; \mathfrak{D}_0). 
		\end{equation*}
	
		\subsection{Conclusion}
		
		Following \cite[Section 4.6]{FeiNovSun}, the final estimate \eqref{final estimate} can be obtained by applying the standard energy method and elliptic estimates to equations \eqref{balance of momentum} and \eqref{balance of internal energy rev }, after differentiating the latter with respect to time; we omit the details as the argument is identical to the one presented in \cite[Section 8]{BasFeiMiz}.
		
		\subsection{Acknowledgement}
		A.~Abbatiello is member of the GNAMPA - INdAM. The work of D. Basari\'{c} was supported by the Czech Sciences Foundation (GA\v CR), Grant Agreement 21--02411S; the Institute of Mathematics of the Academy of Sciences of the Czech Republic is supported by RVO:67985840. The work of N. Chaudhuri is supported by the  ``Excellence Initiative Research University(IDUB)" Programme at University of Warsaw.
		
		\providecommand{\bysame}{\leavevmode\hbox to3em{\hrulefill}\thinspace}
\providecommand{\MR}{\relax\ifhmode\unskip\space\fi MR }
\providecommand{\MRhref}[2]{%
  \href{http://www.ams.org/mathscinet-getitem?mr=#1}{#2}
}
\providecommand{\href}[2]{#2}

\end{document}